\documentclass[10pt]{article}
\usepackage{amscd}
\usepackage{amsfonts}
\usepackage{amsmath}
\usepackage{amssymb}
\usepackage{amsthm}
\usepackage{bbm}
\usepackage{CJK}
\usepackage{fancyhdr}
\usepackage{graphicx}
\usepackage{hyperref}
\usepackage{anyfontsize}
\usepackage{indentfirst}
\usepackage{latexsym}
\usepackage{mathrsfs}
\usepackage{xypic}
\usepackage{tikz-cd}

\newtheorem{theorem}{Theorem}[section]
\newtheorem{lemma}[theorem]{Lemma}
\newtheorem{definition}[theorem]{Definition}
\newtheorem{proposition}[theorem]{Proposition}
\newtheorem{example}[theorem]{Example}

\newtheorem{cor}[theorem]{Corollary}

\usepackage[top=1in,bottom=1in,left=1.25in,right=1.25in]{geometry}
\def\<{\langle}
\def\>{\rangle}
\def\a{\alpha}
\def\b{\beta}

\def\c{\cdot}

\date{}
\begin{document}
\renewcommand{\baselinestretch}{1.2}
\renewcommand{\arraystretch}{1.0}
\title{\bf Deformation cohomology of morphisms of Lie-Yamaguti algebras}
\author{{\bf Bibhash Mondal$^{1}$,     Ripan Saha$^{2}$\footnote
        { Corresponding author (Ripan Saha),  Email: ripanjumaths@gmail.com}}\\
  {\small 1. Department of Mathematics, Behala College}\\
  {\small Behala, 700060, Kolkata, India}\\  
  {\small Email: mondaliiser@gmail.com}\\
 {\small 2. Department of Mathematics, Raiganj University} \\
{\small  Raiganj 733134, West Bengal, India}}
 \maketitle
\begin{center}
\begin{minipage}{13.cm}

{\bf \begin{center} ABSTRACT \end{center}}
We study cohomology of morphisms of Lie-Yamaguti algebras. As an application, we establish that this cohomology `controls' the formal deformations. Additionally, we demonstrate its connection to the abelian extension of morphisms of Lie-Yamaguti algebras.
 \medskip

 \smallskip

 \end{minipage}
 \end{center}
 \normalsize\vskip0.5cm
 
 \section{Introduction}
Lie-Yamaguti algebras, a pivotal concept in non-associative algebra theory, find their roots in the pioneering work of Jacobson \cite{NJ}, who formally introduced the notion of a Lie triple system. This algebraic structure emerged from the context of quantum mechanics \cite{Duffin}, demonstrating its profound relevance in this domain. Nomizu \cite{KN} further extended this theory, establishing a compelling connection between affine connections with parallel torsion and curvature, and invariant connections on reductive homogeneous spaces. Notably, each such space possesses a canonical connection that aligns geodesic parallel translation with the natural group action.

Building on these foundational ideas, K. Yamaguti \cite{KY} introduced the concept of a general Lie triple system: a vector space equipped with both bilinear and trilinear operations, satisfying intriguing relations. This innovation was crucial in characterizing the torsion and curvature tensors associated with Nomizu's canonical connection. The subsequent work of M. Kikkawa \cite{MK75} added a crucial insight, recognizing the intimate relationship between this problem, the canonical connection, and the general Lie triple system defined on the tangent space.

In a transformative contribution, Kikkawa renamed the notion of a general Lie system as a Lie triple algebra, establishing a nomenclature that resonates within the mathematical community. Kinyon and Weinstein \cite{KW} expanded this framework, identifying that Lie triple algebras, which they later termed Lie-Yamaguti algebras, could be constructed from Leibniz algebras. This insight demonstrated the broad applicability of Lie-Yamaguti algebras, which arose initially from Cartan's investigations into Riemannian geometry.

In \cite{ZT}, the author studied infinitesimal deformations of a Lie triple system using the cohomology groups as introduced by K. Yamaguti in \cite{KY-cohomology}. In \cite{ZL}, the authors studied $(2, 3)$-cohomology groups and related deformation theory of Lie-Yamaguti algebras. In \cite{LCM}, the authors discussed the one-paramater formal deformation theory of Lie-Yamaguti algebras, and recently, in \cite{Goswami} the author studied correct obstruction theory to integrate a finite order deformation to a full-blown deformation.

In 1983, Gerstenhaber and Schack \cite{GS83} initiated the study of deformation of morphism of associative algebras. In \cite{ashish}, the author studied deformations of Leibniz algebra morphisms. Recently, Das \cite{das} studied cohomology and deformation theory of Lie algebra morphisms. As Lie-Yamaguti algebra is a generalization (see Example \ref{Lie to LYA} and \ref{Leib to LYA}) of both Lie and Leibniz algebra \cite{Loday1}, it is very natural to study morphism of  Lie-Yamaguti algebra from cohomological point of view. The aim of this paper is to study an algebraic deformation theory of Lie-Yamaguti algebra morphisms. We study the deformation theory of morphism of Lie-Yamaguti algebra following the idea of Gerstenhaber's \cite{G1, G2, G3, G4, G5} algebraic deformation theory. We define a suitable deformation cohomology for such formal deformation. We also show that isomorphism classes of abelian extension of morphism of Lie-Yamaguti algebra has one to one correspondence with the ‘simple’ $(2,3)$- cohomology group.

The paper is organised as follows: In Section \ref{sec2}, we recall the definition and cohomology of Lie-Yamaguti algebras. In Section \ref{sec3}, we discuss representation of morphisms of Lie-Yamaguti algebras. In Section \ref{sec4}, we introduce cohomology of morphism of Lie-Yamaguti algebra which involves both Lie-Yamaguti part and the morphism part. In Section \ref{sec5}, we discuss one-parameter formal deformation theory of morphism of Lie-Yamaguti algebra and show that the deformation is `controlled' by our cohomology defined in Section \ref{sec4}. In Section \ref{sec6}, we study abelian extension of morphism of Lie-Yamaguti algebra and its relationship with the cohomology. Finally, in Section \ref{sec7}, we conclude our paper by pointing out that the work of this paper is a special case of a more general problem of diagram of Lie-Yamaguti algebras.

\section{Lie-Yamaguti algebra and its cohomology} \label{sec2}
In this section, we recall the cohomology of Lie-Yamaguti algebra as in \cite{KY, KY-LTS, KY-cohomology}, and \cite{Goswami}.
Let $(L, [~, ~], \{~, ~, ~\})$ be a Lie-Yamaguti algebra with representation $V$.
\begin{definition}
A Lie-Yamaguti  algebra  is a triple $(L, [~, ~], \{~, ~, ~\})$, where $L$ is a $k$-vector space, $[~, ~]$ is a binary operation and $\{~, ~, ~\}$ is 
a ternary operation on $L$ such that
\begin{eqnarray*}
&&(1)~~ [x, y]=-[y,  x],\\
&&(2) ~~\{x,y,z\}=-\{y,x,z\},\\
&&(3) ~~\circlearrowleft_{(x,y,z)}([[x, y], z]+\{x, y, z\})=0,\\
&&(4) ~~\circlearrowleft_{(x,y,z)}(\{[x, y], z, u\})=0,\\
&&(5) ~~\{x, y, [u, v]\}=[\{x, y, u\}, v]+([u,\{x, y, v\}],\\
&&(6) ~~ \{x, y, \{u, v, w\}\}=\{\{x, y, u\}, v, w\} +\{u,  \{x, y, v\}, w\} +\{u,  v, \{x, y, w\}\},
\end{eqnarray*}
for all $ x, y, z, u, v, w\in L$ and where $\circlearrowleft_{(x,y,z)}$ denotes the sum over cyclic permutation of $x, y, z$.
\end{definition}

\begin{example}\label{Lie to LYA}
Let $(\mathfrak g, [~, ~])$ be a Lie algebra over $k$. Then, $\mathfrak g$ has a Lie-Yamaguti algebra structure induced by the given Lie bracket, the trilinear operation being:
$$\{a, b, c\} = [[a, b], c]$$ for all $a, b, c \in \mathfrak g$.
\end{example}

\begin{example}\label{Leib to LYA}
Let $(\mathfrak g, \cdot)$ be a Leibniz algebra. Consider a bilinear operation and a trilinear operation as follows:
$$[~, ~]: \mathfrak g \times \mathfrak g \to \mathfrak g, \quad [a, b] := a\cdot b-b\cdot a,~~ a, b \in \mathfrak g;$$
$$\{~, ~, ~\}:  \mathfrak g \times \mathfrak g \times \mathfrak g \to \mathfrak g, \quad \{a, b, c\} := -(a\cdot b)\cdot c,~~a, b, c \in \mathfrak g.$$ Then, $(\mathfrak g, [~,~],\{~,~,~\})$ is a Lie-Yamaguti algebra. 
\end{example}

\begin{definition} Let $(L, [~, ~], \{~, ~, ~\})$  be a Lie-Yamaguti  algebra and $V$ be a vector space. A representation of
$L$ on $V$ consists of a linear map $\rho: L\rightarrow$ End($V$) and two bilinear maps $D, \theta: L\times L\rightarrow$ End($V$)
satisfying the following conditions
\begin{eqnarray*}
 &&(1) ~D(x, y)-\theta(y, x)+\theta(x, y)+\rho([x, y])-\rho(x)\rho(y)+\rho(y)\rho(x)=0,\\
 &&(2) ~D([x, y], z)+D([y, z],x)+D([z, x], y)=0, \\
 &&(3)  ~\theta([x, y], z)=\theta(x, z)\rho(y)-\theta(y, z)\rho(x), \\
 &&(4) ~D(x, y)\rho(z)=\rho(z)D(x, y)+\rho(\{x, y, z\}), \\
 &&(5)~\theta(x, [y, z])=\rho(y)\theta(x, z)-\rho(z)\theta(x, y),\\
 &&(6) ~D(x, y)\theta(u, v)=~\theta(u, v)D(x, y)+\theta(\{x, y, u\}, v)+\theta(u, \{x, y, v\}),\\
 &&(7)  ~\theta(x, \{y, z, u\})=\theta(z, u) \theta(x, y)-\theta(y, u) \theta(x, z)+D(y, z)\theta(x, u),
\end{eqnarray*}
for any $x, y, z, u, v\in L$. Here we say the $(\rho,D,\theta)$ is a representation of $L$ on $V.$
\end{definition}

 Let us define the cohomology groups of $L$ with coefficients
in $V$. Let $f: L\times L \times \c\c\c\times L$ be an $n$-linear map of $L$ into $V$ such that the following conditions are
satisfied:
\begin{eqnarray*}
 f(x_1, \ldots, x_{2i-1}, x_{2i}, \ldots, x_{n})=0, ~~~~\mbox{if}~~~ x_{2i-1}=x_{2i}.
\end{eqnarray*}
The vector space spanned by such linear maps is called an $n$-cochain of $L$, which is denoted by
$C^{n}(L, V)$ for $n\geq 1$.
 
\begin{definition}
For any $(f, g)\in  C^{2n}(L, V) \times C^{2n+1}(L, V )$, $n\geq 1$ the coboundary operator  $\delta: (f, g)\rightarrow (\delta_I f, \delta_{II} g)$ is a mapping from $C^{2n}(L, V) \times C^{2n+1}(L, V )$ into $C^{2n+2}(L, V) \times C^{2n+3}(L, V )$ defined as follows:
\begin{eqnarray*}
&& (\delta_If)(x_1,x_2,\ldots,x_{2n+2})\\
&=&\rho(x_{2n+1})g(x_1,x_2,\ldots,x_{2n},x_{2n+2})-\rho(x_{2n+2})g(x_1,x_2,\ldots,x_{2n},x_{2n+1})\\
  &&-g(x_1,x_2,\ldots,x_{2n},[x_{2n+1},x_{2n+2}])\\
  &&+\sum_{k=1}^{n}(-1)^{n+k+1}D(x_{2k-1},x_{2k})f(x_{1},\ldots,\widehat{x}_{2k-1},\widehat{x}_{2k},\ldots,x_{2n+2})\\
   &&+\sum_{k=1}^{n}\sum_{j=2k+1}^{2n+2}(-1)^{n+k}f(x_{1},\ldots,\widehat{x}_{2k-1},\widehat{x}_{2k},\ldots,\{x_{2k-1},x_{2k},x_j\},\ldots,x_{2n+2}),\\
&&(\delta_{II}g)(x_1,x_2,\ldots,x_{2n+3})\\
&=&\theta(x_{2n+2},x_{2n+3})g(x_1,\ldots,x_{2n+1})-\theta(x_{2n+1},x_{2n+3})g(x_1,\ldots,x_{2n},x_{2n+1})\\
  &&+\sum_{k=1}^{n+1}(-1)^{n+k+1}D(x_{2k-1},x_{2k})g(x_{1},\ldots,\widehat{x}_{2k-1},\widehat{x}_{2k},\ldots,x_{2n+3})\\
   &&+\sum_{k=1}^{n+1}\sum_{j=2k+1}^{2n+3}(-1)^{n+k}g(x_{1},\ldots,\widehat{x}_{2k-1},\widehat{x}_{2k},\ldots,\{x_{2k-1},x_{2k},x_j\},\ldots,x_{2n+3}).
\end{eqnarray*}
\end{definition}
We define $ C^{(2n,2n+1)}(L, V) :=C^{2n}(L, V) \times C^{2n+1}(L, V)$. Then $\{C^{(\bullet,\bullet)}(L,V), \delta\}$ is a cochain complex.\\
Let $C^1(L,V)$ denote the space of all linear maps from $L$ into $V$. We define the subspace $C(L,V)$ by the diagonal elements $(f,f)\in C^1(L,V) \times C^{1}(L,V.)$.
  The map $\delta : C(L,V) \rightarrow C^{(2,3)}(L,V)$ is defined by $\delta(f,f)=(\delta_If,\delta_{II}f)$, where
  \begin{align*}
 & \delta_If(a,b)=\rho(a)f(b)-\rho(b)f(a)-f([a,b]) ,\\
  &\delta_{II}(a,b,c)=\theta(b,c)f(a)-\theta(a,c)f(b)+D(a,b)f(c)-f(\{a,b,c\})
  \end{align*}
for all $a,b,c \in L.$
From \cite{KY-cohomology}, we have for any $f\in C^1(L,V)$ , $\delta_I\delta_If=\delta_{II}\delta_{II}f=0$ and In general for any $(f,g)\in C^{(2p,2p+1)}(L,V)~\mbox{we have}~~(\delta \circ \delta)(f,g)=(\delta_I \circ \delta(f),\delta_{II}\circ \delta_{II}(g))=0.$

 \section{Morphism of Lie-Yamaguti algebra and its representation}\label{sec3}
 In this section, we discuss momorphism between Lie-Yamaguti algebras and representation of Lie-Yamaguti algebra as defined in \cite{KY, KY-cohomology}.

\begin{definition}
A momorphism between two Lie-Yamaguti  algebras  $(L, [~, ~], \{~, ~, ~\})$ and $(L', [~, ~]', \{~, ~, ~\}')$ is a linear map $\varphi: L\rightarrow L'$ satisfying
\begin{eqnarray*}
\varphi([x, y])=[\varphi(x), \varphi(y)]', ~~~~~~~~~\varphi(\{x, y, z\})=\{\varphi(x), \varphi(y),\varphi(z)\}'
\end{eqnarray*}
for all $x,y,z \in L.$
\end{definition}

\begin{definition}
Let $\phi : L_1 \rightarrow L_2$ and $\phi^{'} : L_1^{'} \rightarrow L_2^{'}$ be two morphism of Lie-Yamaguti algebras. A homomorphism between these two morphism of Lie-Yamaguti algebra is a pair $(\alpha , \beta)$, where $\alpha : L_1 \rightarrow L_1^{'}$ and $\beta : L_2 \rightarrow L_2^{'}$ are two Lie-Yamaguti algebra homomorphism satisfying the condition $\phi^{'} \circ \alpha = \beta \circ \phi$.
\end{definition}
\begin{definition}
Let $\phi : L_1 \rightarrow L_2$ be a morphism of Lie-Yamaguti algebra. A representation of it is a ordered triple $(V,W,\psi)$, where $V$, $W$ are representations of $L_1$ and $L_2$ respectively, and $\psi : V \rightarrow W $ is a linear operator satisfying the following conditions
\begin{align*}
&\psi(\rho _V(x)v)=\rho_W(\phi(x))\psi(v)\\
& \psi(D_V(x,y)v)=D_W(\phi(x),\phi(y))\psi(v)\\
& \psi(\theta_V(x,y)v)=\theta_W(\phi(x),\phi(y))\psi(v)\\
\end{align*}
for all $x,y \in L_1$ and $v \in V.$
\end{definition}
The following proposition gives an example of representation of a Lie-Yamaguti algebra induced by representation of morphism of Lie-Yamaguti algebra.
\begin{proposition}
Let $\phi : L_1 \rightarrow L_2$ be a morphism of Lie-Yamaguti algebra and $(V,W,\psi)$ be a representation of it. Now the maps 
\begin{align*}
&\rho_W^{L_1}: L_1 \rightarrow End(W) \\
& D_W^{L_1} : L_1 \times L_1 \rightarrow End(W)\\
& \theta_W^{L_1} : L_1 \times L_1 \rightarrow End(W)
\end{align*}
 defined by  
\begin{align*}
& \rho_W^{L_1}(x)v=\rho_W (\phi(x))v \\
& \theta_W^{L_1}(x,y)v= \theta _W(\phi(x),\phi(y))v\\
& D_W^{L_1}(x,y)v= D _W(\phi(x),\phi(y))v
\end{align*}
respectively for all $x,y \in L_1$ and $v \in V$ makes
$(\rho _W^{L_1}, D_W^{L_1},\theta_W^{L_1})$ is representation of $L_1$ on $W.$ 
\end{proposition}
\begin{proof}
Let $x,y,z \in L_1$ and $w \in W.$ Observe that
\begin{align*}
&D_W^{L_{1}}(x,y)w-\theta _W^{L_{1}}(y,x)w+\theta _W^{L_{1}}(x,y)w+\rho _W^{L_{1}}([x,y])w\\
&-\rho _W^{L_{1}}(x)\rho _W^{L_{1}}(y)w+\rho _W^{L_{1}}(y)\rho _W^{L_{1}}(x)w\\
&=D_W(\phi(x),\phi(y))w-\theta_W(\phi(y),\phi(x))w+\theta_W(\phi(x),\phi(y))w+\rho_W([\phi(x),\phi(y)])w\\
&-\rho _W^{L_{1}}(x)\rho_W(\phi(y))w+\rho _W^{L_{1}}(y)\rho_W(\phi(x))w\\
&=D_W(\phi(x),\phi(y))w-\theta_W(\phi(y),\phi(x))w+\theta_W(\phi(x),\phi(y))w+\rho_W([\phi(x),\phi(y)])w\\
&-\rho _W(\phi(x))\rho_W(\phi(y))w+\rho(\phi(y))\rho_W(\phi(x))w\\
&=0,~~~\mbox{as}~ W ~\mbox{is a representation of}~ L_2. 
\end{align*}
On the other hand, we have
\begin{align*}
&D_W^{L_{1}}([x,y],z)w+D_W^{L_{1}}([y,z],x)w+D_W^{L_{1}}([z,x],y)\\
&=D_W(\phi([x,y]),\phi(z))w+D_W(\phi([y,z]),\phi(x))w+D_W(\phi([z,x]),\phi(y))w\\
&=D_W([\phi(x),\phi(y)],\phi(z))w+D_W([\phi(y),\phi(z)],\phi(x))w+D_W([\phi(z),\phi(x)],\phi(y))w\\
&=0,~~~\mbox{as}~ W ~\mbox{is a representation of}~ L_2. 
\end{align*}

The other axioms of representation of a Lie-Yamaguti algebra can be proved similarly as above. Hence, $(\rho _W^{L_1}, D_W^{L_1},\theta_W^{L_1})$ is a representation of $L_1$ on $W.$ 
\end{proof}
We denote the above representation by $W_{\phi}.$ The following proposition shows that given a homomorphism between two morphism of Lie-Yamaguti algebra induces a representation of a morphism of Lie-Yamaguti algebra.
\begin{proposition}
Let $\phi : L_1 \rightarrow L_2$, $\phi^{'} : L_1^{'} \rightarrow L_2^{'}$ be two morphism of Lie Yamaguti algebras and $(\alpha, \beta)$ be a homomorphism between them. Then, $( L_1^{'}, L_2^{'}, \phi^{'})$  is a representation of Lie-Yamaguti algebra morphism $\phi : L_1 \rightarrow L_2$, where $\rho_{L_1^{'}} (x): L_1 \rightarrow End(L_1^{'}), ~D_{L_1^{'}}(x,y):L_1 \times L_1 \rightarrow End(L_1^{'}),~ \theta_{L_1^{'}}(x,y):L_1 \times L_1 \rightarrow End(L_1^{'})$  are defined by 
\begin{align*}
&\rho(x)x^{'}=[\alpha(x),x^{'}]\\
&D(x,y)x^{'}=\{\alpha(x),\alpha(y),x^{'}\}\\
&\theta(x,y)x^{'}=\{x^{'},\alpha(x),\alpha(y)\}
\end{align*}
respectively for all $x,y \in L_1$ and $x^{'}\in L_1^{'}$. Also, $\rho_{L_2^{'}} (x): L_2 \rightarrow End(L_2^{'}),~ D_{L_2^{'}}(x,y):L_2 \times L_2 \rightarrow End(L_2^{'}),~ \theta_{L_2^{'}}(x,y):L_2 \times L_2 \rightarrow End(L_2^{'})$  are defined by 
\begin{align*}
&\rho(x)x^{'}=[\beta(x),x^{'}]\\
&D(x,y)x^{'}=\{\beta(x),\beta(y),x^{'}\}\\
&\theta(x,y)x^{'}=\{x^{'},\beta(x),\beta(y)\}
\end{align*}
respectively for all $x,y \in L_2$ and $x^{'}\in L_2^{'}$.
\end{proposition}
\begin{proof}
For any $x,y,z \in L_1$ and $x^{'}\in L_1^{'}$, we have
\begin{align*}
&D(x,y)x^{'}-\theta(y,x)x^{'}+\theta(x,y)x^{'}+\rho([x,y])x^{'}-\rho(x)\rho(y)x^{'}+\rho(y)\rho(x)x^{'}\\
&=\{\alpha (x),\alpha(y),x^{'}\}-\{x^{'},\alpha(y),\alpha(x)\}+\{x^{'},\alpha(x),\alpha(y)\}\\
&+[[\alpha(x),\alpha(y)],x^{'}]-[\alpha(x),[\alpha(y),x^{'}]]+[\alpha(y),[\alpha(x),x^{'}]]\\
&=\{\alpha (x),\alpha(y),x^{'}\}+\{\alpha(y),x^{'},\alpha(x)\}+\{x^{'},\alpha(x),\alpha(y)\}\\
&+[[\alpha(x),\alpha(y)],x^{'}]+[[\alpha(y),x^{'}],\alpha(x)]+[[x^{'},\alpha(x)],\alpha(y)]\\
&=0.
\end{align*}
Again,
\begin{align*}
&D([x,y],z)x^{'}+D([y,z],x)x^{'}+D([z,x],y)x^{'}\\
&=\{\alpha([x,y]),\alpha(z),x^{'}\}+\{\alpha([y,z]),\alpha(x),x^{'}\}+\{\alpha([z,x]),\alpha(y),x^{'}\}\\
&=\{[\alpha (x),\alpha (y)],\alpha(z),x^{'}\}+\{[\alpha(y),\alpha(z)],\alpha(x),x^{'}\}+\{[\alpha(z),\alpha(x)],\alpha(y),x^{'}\}\\
&=0.
\end{align*}
In a similar way other axioms of representation of Lie-Yamaguti algebra can be shown.
Therefore, $L_1^{'}$ is a representation of $L_1$. In a similar way as above $L_2^{'}$ is a representation of $L_2.$
Now, \[\phi^{'}(\rho_{L_1^{'}}(x)x^{'})=\phi^{'}([\alpha(x),x^{'}])=[\phi^{'}\circ \alpha (x),\phi^{'}(x^{'})]\\
=[\beta \circ \phi^{'}(x^{'}),\phi^{'}(x^{'})]=\rho_{L_2^{'}}(\phi(x))\phi^{'}(x^{'})\]
The other two conditions for representation for morphism of Lie-Yamaguti algebra can be shown same way as above.
Hence, $ (L_1^{'}, L_2^{'}, \phi)$  is a representation of Lie-Yamaguti algebra morphism $\phi : L_1 \rightarrow L_2.$ 
\end{proof}

\section{Cohomology of morphism of Lie-Yamaguti algebras}\label{sec4}
In this section, we define the cohomology of morphism of Lie-Yamaguti algebra using cohomolgy theory of Lie-Yamaguti algebra motivated by the cohomolgoy theory of morphism of different algebras in \cite{GS83, ashish, das}.\\
Let $\phi : L_1 \rightarrow L_2$ be a morphism of Lie-Yamaguti algebra and $(V,W,\psi)$ be a representation of it. Now we have the following three cochain complexes
\begin{itemize}
\item $\{C^{(\bullet,\bullet)}(L_1,V),\delta^{'}\}$, the cochain complex of $L_1$ with coefficients in $V.$
\item $\{C^{(\bullet,\bullet)}(L_2,W),\delta^{''}\}$, the cochain complex of $L_2$ with coefficients in $W.$
\item $\{C^{(\bullet,\bullet)}(L_1,W_{\phi}),\delta^{'''}\}$, the cochain complex of $L_1$ with coefficients in $W_\phi.$
\end{itemize}

 Now, using the above three cochain complex we define the cohomology of the morphism of Lie Yamaguti algebra $\phi : L_1 \rightarrow L_2$ with coefficients in $(V,W,\psi).$

 For each $n \geq 1$, we define the $(2n,2n+1)$-th cochain group $C^{(2n,2n+1)}_{mLYA}(\phi, \psi)$  by 
 
$$C^{(2n,2n+1)}_{mLYA}(\phi, \psi)
=
  \begin{cases}
& C^{(2,3)}(L_1,V)\oplus C^{(2,3)}(L_2,W) \oplus C(L_1,W_{\phi}) ~~\mbox{if}~~n=1\\
& C^{(2n,2n+1)}(L_1,V) \oplus C^{(2n,2n+1)}(L_2,W) \oplus C^{(2n-2,2n-1)}(L_1,W_{\phi}),~~~\mbox{if} ~~n\geq 2.
\end{cases}$$

We define the coboundary map 
$\delta_{mLYA} : C^{(2n,2n+1)}_{mLYA}(\phi, \psi) \rightarrow C^{(2n+2,2n+3)}_{mLYA}(\phi, \psi)$  by
\[
  \delta_{mLYA} (\alpha , \beta , \gamma) =
  \begin{cases}
   (\delta(\alpha), \delta(\beta), \phi \circ \alpha -\beta \phi -\delta (\gamma) )  & \text{for } (\alpha , \beta , \gamma) \in C^{(2,3)}_{mLYA}(\phi, \psi), n=1 \\
    (\delta^{'}(\alpha), \delta^{''}(\beta), \psi \circ \alpha- \beta \circ \wedge^n \phi - \delta^{'''}(\gamma)) & \text{for } (\alpha , \beta , \gamma) \in C^{(2n,2n+1)}_{mLYA}(\phi, \psi), n\geq 2, 
  \end{cases} 
\]
 where  
$\delta: C(L_1,W_{\phi}) \rightarrow C^{(2,3)}(L_1, W_{\phi})$ is defined by
$\delta (\gamma)=(\delta_I\gamma ,\delta_{II}\gamma)$ such that 
\begin{align*}
&\delta_I \gamma(x,y)=[\phi(x), \gamma (y)]+ [\gamma (x), \phi(y)]-\gamma([x,y]) \\
& \mbox{and}\\
& \delta_{II}\gamma (x,y,z)= \{\gamma (x), \phi(y),\phi (z)\}+\{\phi(x), \gamma (y), \phi (z)\}+\{\phi (x),\phi (y), \gamma(z)\}-\gamma ( \{x,y,z\})
\end{align*} 
for all $x,y,z \in L_1.$ Also, note that $\beta \phi :L_1^{\otimes n} \rightarrow W_{\phi}$ is a linear function defined by 
$$\beta \phi (x_1,x_2,\ldots,x_n)=\beta (\phi(x_1),\phi(x_2), \ldots, \phi(x_n))$$ for all $x_1,x_2,\ldots,x_n \in L_1.$

\begin{proposition}
The map $\delta_{mLYA}$ defined above satisfies $\delta_{mLYA}^2=0.$
\end{proposition}

\begin{proof}
\textbf{Case-I:} \\For $n=1,$ and $(\alpha, \beta, \gamma) \in C^{(2,3)}_{mLYA}(\phi, \psi)$, we have
\begin{align*}
&d^2(\alpha,\beta,\gamma)\\
&=d(\delta(\alpha), \delta(\beta), \phi \circ \alpha -\beta \phi -\delta (\gamma))\\
&=(\delta^2(\alpha), \delta^2(\beta),\phi(\delta(\alpha))-\delta (\beta)\phi -\delta (\phi \circ \alpha - \beta \phi- \delta(\gamma)))\\
&= (\delta^2(\alpha), \delta^2(\beta),\phi(\delta(\alpha))-\delta (\beta)\phi -\delta (\phi \circ \alpha) - \delta(\beta \phi))\\
&=0,~ \mbox{since}~ \delta^2=0 ~ by ~\cite{KY}. 
\end{align*}

\textbf{Case-II:} For $n\geq 2.$\\
For $(\alpha , \beta , \gamma) \in C^{(2n,2n+1)}_{mLYA}(\phi, \psi)$, we have 
\begin{align}
&\delta_{mLYA}^2(\alpha, \beta ,\gamma)\nonumber \\
&=(\delta^{'}(\alpha), \delta^{''}(\beta), \psi \circ \alpha - \beta \circ \wedge^n \phi - \delta^{'''}(\gamma)) \nonumber \\
&=((\delta^{'})^2(\alpha),(\delta^{''})^2(\beta),\psi \circ \delta^{'}(\alpha)-\delta^{''}(\beta) \circ \wedge^{n+1}\phi -\delta^{'''}(\psi \circ \alpha)+\delta^{'''}(\beta \circ \wedge^n \phi)+(\delta^{'''})^2(\gamma)) \nonumber\\
&=(0,0, \psi \circ \delta^{'}(\alpha)-\delta^{''}(\beta) \circ \wedge^{n+1}\phi -\delta^{'''}(\psi \circ \alpha)+\delta^{'''}(\beta \circ \wedge^n \phi)) \label{Eq1}\
\end{align}

For $ x_1,x_2,\ldots,x_{2n+2} \in L_1$, observe that
\begin{align}
&\delta^{'''}(\psi \circ \alpha)(x_1,x_2,\ldots,x_{2n+2}) \nonumber\\
&=\delta^{'''}(\psi \circ \alpha_I,\psi \circ \alpha_{II})(x_1,x_2,\ldots,x_{2n+2}) \nonumber \\
&=(\delta^{'''}_I(\psi \circ \alpha_I),\delta^{'''}_{II}(\psi \circ \alpha_{II}))(x_1,x_2,\ldots,x_{2n+2}) \nonumber \label{Eq2}.
\end{align}

Again,
\begin{align*}
&\delta^{'''}_I(\psi \circ \alpha_I)(x_1,x_2,\ldots,x_{2n+2})\\
&=\rho_{W}^{L_1}(x_{2n+1})(\psi \circ \alpha_{II})(x_1,x_2,\ldots,x_{2n},x_{2n+2})-\rho_{W}^{L_1}(x_{2n+2})(\psi \circ \alpha_{II})(x_1,x_2,\ldots,x_{2n},x_{2n+1})\\
  &-(\psi \circ \alpha_{II})(x_1,x_2,\ldots,x_{2n},[x_{2n+1},x_{2n+2}])\\
  &+\sum_{k=1}^{n}(-1)^{n+k+1}D_W^{L_1}(x_{2k-1},x_{2k})(\psi \circ \alpha_{I})(x_{1},\ldots,\widehat{x}_{2k-1},\widehat{x}_{2k},\ldots,x_{2n+2})\\
   &+\sum_{k=1}^{n}\sum_{j=2k+1}^{2n+2}(-1)^{n+k}(\psi \circ \alpha_{I})(x_{1},\ldots,\widehat{x}_{2k-1},\widehat{x}_{2k},\ldots,\{x_{2k-1},x_{2k},x_j\},\ldots,x_{2n+2}),\\
&=\psi(\rho_V(x_{2n+1})\alpha_{II}(x_1,x_2,\ldots,x_{2n},x_{2n+2}))- \psi (\rho_V(x_{2n+2})\alpha_{II}(x_1,x_2,\ldots,x_{2n},x_{2n+1}))\\
  &-\psi (\alpha_{II}(x_1,x_2,\ldots,x_{2n},[x_{2n+1},x_{2n+2}]))\\
  &+\sum_{k=1}^{n}(-1)^{n+k+1}\psi_V(D_V(x_{2k-1},x_{2k}) \alpha_{I}(x_{1},\ldots,\widehat{x}_{2k-1},\widehat{x}_{2k},\ldots,x_{2n+2}))\\
   &+\sum_{k=1}^{n}\sum_{j=2k+1}^{2n+2}(-1)^{n+k}\psi ( \alpha_{I}(x_{1},\ldots,\widehat{x}_{2k-1},\widehat{x}_{2k},\ldots,\{x_{2k-1},x_{2k},x_j\},\ldots,x_{2n+2})),\\
   &=(\psi \circ \delta^{'}(\alpha_I))(x_1,x_2,\ldots,x_{2n+2}).
 \end{align*}
Thus, we have 
\[\delta^{'''}_I(\psi \circ \alpha_I)= (\psi \circ \delta^{'}(\alpha_I))\]
Similarly, it can be proved that
\[\delta^{'''}_{II}(\psi \circ \alpha_{II})= (\psi \circ \delta^{'}(\alpha_{II})).\]

Now from the Equation (\ref{Eq2}), we have

\begin{align*}
&\delta^{'''}(\psi \circ \alpha)(x_1,x_2,\ldots,x_{2n+2}) \nonumber\\
&=(\delta^{'''}_I(\psi \circ \alpha_I),\delta^{'''}_{II}(\psi \circ \alpha_{II}))(x_1,x_2,\ldots,x_{2n+2}) \\
&= (\psi \circ \delta^{'}(\alpha_I),\psi \circ \delta^{'}(\alpha_{II}))(x_1,x_2,\ldots,x_{2n+2})\\
&=(\psi \circ \delta^{'}(\alpha))(x_1,x_2,\ldots,x_{2n+2}).
\end{align*}
 Hence, $\delta^{'''}(\psi \circ \alpha)=(\psi \circ \delta^{'}(\alpha)).$

Now, we want to show that
\[\delta ^{'''}(\beta \circ \wedge ^{n} \phi)= \delta ^{''}(\beta)\circ \wedge^{n+1} (\phi)\]
Note that \begin{align}
&\delta^{'''}(\beta \circ \wedge^n \phi )(x_1,x_2,\dots ,x_{2n+2}) \nonumber\\
&=\delta^{'''}(\beta_I \circ \wedge^n \phi,\beta_{II} \circ \wedge^{n} \phi ) (x_1,x_2,\dots ,x_{2n+2}) \nonumber\\
&=(\delta^{'''}_I(\beta_I \circ \wedge^n \phi), \delta_{II}^{'''}(\beta_{II} \circ \wedge^{n} \phi))(x_1,x_2,\dots ,x_{2n+2}).\label{Eq3}
\end{align}
Again, we have\begin{align*}
&\delta^{'''}_I(\beta_I \circ \wedge^n \phi)(x_1,x_2,\dots ,x_{2n+2})\\
&=\rho_{W}^{L_1}(x_{2n+1})(\beta_{II} \circ \wedge^n \phi)(x_1,x_2,\ldots,x_{2n},x_{2n+2})-\rho_{W}^{L_1}(x_{2n+2})(\beta_{II} \circ \wedge^n \phi)(x_1,x_2,\ldots,x_{2n},x_{2n+1})\\
  &-(\beta_{II} \circ \wedge^n \phi)(x_1,x_2,\ldots,x_{2n},[x_{2n+1},x_{2n+2}])\\
  &+\sum_{k=1}^{n}(-1)^{n+k+1}D_W^{L_1}(x_{2k-1},x_{2k})(\beta_I \circ \wedge^n \phi)(x_{1},\ldots,\widehat{x}_{2k-1},\widehat{x}_{2k},\ldots,x_{2n+2})\\
   &+\sum_{k=1}^{n}\sum_{j=2k+1}^{2n+2}(-1)^{n+k}(\beta_I \circ \wedge^n \phi)(x_{1},\ldots,\widehat{x}_{2k-1},\widehat{x}_{2k},\ldots,\{x_{2k-1},x_{2k},x_j\},\ldots,x_{2n+2}),\\
&=\rho_W(\phi(x_{2n+1}))\beta_{II}(\phi (x_1),\phi(x_2),\cdots,\phi(x_{2n}),\phi(x_{2n+2}))\\
&-\rho_W(\phi (x_{2n+2})) \beta_{II}(\phi(x_1),\phi(x_2), \cdots, \phi(x_{2n}),\phi(x_{2n+1}))\\
&-\beta_{II}(\phi (x_1), \phi(x_2),\cdots,\phi(x_{2n}),[\phi (x_{2n+1}),\phi (x_{2n+2})])\\
& +\sum_{k=1}^{n}(-1)^{n+k+1}D_W(\phi (x_{2k-1}),\phi(x_{2k}))\beta_{I}(\phi(x_1),\cdots, \widehat{\phi(x_{2k-1})},\widehat{\phi(x_{2k})},\ldots,\phi (x_{2n+2}))\\
&+\sum_{k=1}^{n}\sum_{j=2k+1}^{2n+2}(-1)^{n+k}\beta_I(\phi(x_1),\ldots,\widehat{\phi(x_{2k-1})},\widehat{\phi(x_{2k})},\ldots,\{\phi (x_{2k-1}),\phi (x_{2k}),\phi(x_j)\},\ldots,\phi(x_{2n+2}))\\
&=\delta_I^{''}(\beta_I)\circ \wedge^{n+1}(\phi)(x_1,x_2,\dots ,x_{2n+2}).
\end{align*}
Hence, we have \[\delta^{'''}_I(\beta_I \circ \wedge^n \phi)=\delta_I^{''}(\beta_I)\circ \wedge^{n+1}(\phi).\]
Similarly, we have \[\delta^{'''}_{II}(\beta_{II}\circ \wedge^n \phi)=\delta_{II}^{''}(\beta_{II})\circ \wedge^{n+1}(\phi).\]

Now from the Equation (\ref{Eq3}), we have
\begin{align*}
&\delta^{'''}(\beta \circ \wedge^n \phi )(x_1,x_2,\dots ,x_{2n+2})\\
&=(\delta^{'''}_I(\beta_I \circ \wedge^n \phi), \delta_{II}^{'''}(\beta_{II} \circ \wedge^{n} \phi))(x_1,x_2,\dots ,x_{2n+2})\\
&=(\delta_I^{''}(\beta_I)\circ \wedge^{n+1}(\phi),\delta_{II}^{''}(\beta_{II})\circ \wedge^{n+1}(\phi))(x_1,x_2,\dots ,x_{2n+2})\\
&=\delta ^{''}(\beta)\circ \wedge^{n+1} (\phi)(x_1,x_2,\dots ,x_{2n+2}).
\end{align*}
Hence, we have \[\delta ^{'''}(\beta \circ \wedge ^{n} \phi)= \delta ^{''}(\beta)\circ \wedge^{n+1} (\phi).\]
Therefore, from the Equation  (\ref{Eq1}), we have $\delta_{mLYA}^2=0.$
\end{proof}
Thus, it follows from the above proposition that $\{C^{(\bullet, \bullet)}_{mLYA},\delta_{mLYA}\}$ is a cochain complex. Let the space of $(2n,2n+1)$-th cocyles are denoted by $Z^{(2n,2n+1)}_{mLYA}(\phi, \psi)$ and the space of $(2n,2n+1)$-th coboundaries are denoted by $B^{(2n,2n+1)}_{mLYA}(\phi, \psi)$. Observe that, we have $B^{(2n,2n+1)}_{mLYA}(\phi, \psi) \subset Z^{(2n,2n+1)}_{mLYA}(\phi, \psi).$ The corresponding quotient groups
\[H^{(2n,2n+1)}_{mLYA}(\phi, \psi):= \frac{Z^{(2n,2n+1)}_{mLYA}(\phi, \psi)}{B^{(2n,2n+1)}_{mLYA}(\phi, \psi)}~~,~\mbox{for}~~ n \geq 1,\]
are called the cohomology groups of the morphism of Lie-Yamaguti algebra $\phi : L_1 \rightarrow L_2$ with coefficients in the representation $(V,W,\psi).$
\section{Deformation of morphisms of Lie-Yamaguti algebra}\label{sec5}
In this section we develop one-parameter formal deformation theory of morphisms of Lie-Yamaguti algebra over a field $k$ of $Char(k)=0.$ Let $K=k[[t]]$ be the formal power series ring with coefficient in $k$. For a Lie-Yamaguti algebra $L$ define $L[[t]]:=L \otimes _{k}K$, then $L[[t]]$ is a module over $K.$

\begin{definition}
Let $\phi: L_1 \rightarrow L_2$ be a morphism of Lie Yamaguti algebra. A formal one parameter deformation of it is a order triple $((f_t,g_t),(f_t^{'},g_t^{'}),(\phi_t,\phi_t))$ , with the formal power series of the forms 
\begin{align*}
&f_t=[~,~]+\sum _{i \geqslant 1} f_i t^i ; f^{'}_t=[~,~]+\sum _{i \geqslant 1} f^{'}_i t^i \\
&g_t=\{~,~,~\}+\sum _{i \geqslant 1} g_i t^i ; g^{'}_t=\{~,~,~\}+\sum _{i \geqslant 1} g^{'}_i t^i \\
&\phi_t=\sum _{i=0}\phi_i t^{'}, \phi_0=\phi,
\end{align*}
 where for $i \geq 0$, $f_i : L_1 \times L_1 \rightarrow L_1$ and  $f^{'}_i : L_2 \times L_2 \rightarrow L_2$ are $k$- bilinear map (extended to be $K$- bilinear) and also, for $i \geq 0$, $g_i : L_1\times \times L_1 \rightarrow L_1$ and  $g^{'}_i : L_2 \times L_2 \times L_2 \rightarrow L_2$ are $k$- trilinear map (extended to be $K$- trilinear) such that the following conditions are satisfied
\begin{enumerate}
\item $(L_1[[t]],f_t,g_t)$ is a Lie-Yamaguti algebra.
\item $(L_2[[t]],f^{'}_t,g^{'}_t)$ is a Lie-Yamaguti algebra.
\item $\phi_t : L_1[[t]] \rightarrow L_2[[t]] $ is a morphism of Lie-Yamaguti algebra.
\end{enumerate}
\end{definition}
\begin{definition}
The $(2,3)$- cochain  $((f_1,g_1),(f_1^{'},g_1^{'}),(\phi_1,\phi_1))$ is called the infinitesimal of the deformation  $((f_t,g_t),(f_t^{'},g_t^{'}),(\phi_t,\phi_t))$. More generally  $((f_i,g_i),(f_i^{'},g_i^{'}),(\phi_i,\phi_i))=0$ for $1\leq i \leq n-1$ and  $((f_n,g_n),(f_n^{'},g_n^{'}),(\phi_n,\phi_n))$ is non zero, then  $((f_n,g_n),(f_n^{'},g_n^{'}),(\phi_n,\phi_n))$ is called the $n$- infinitesimal of the deformation  $((f_t,g_t),(f_t^{'},g_t^{'}),(\phi_t,\phi_t))$.
\end{definition}

Let $((f_t,g_t),(f_t^{'},g_t^{'}),(\phi_t,\phi_t))$ be a formal deformation of a morphism of Lie-Yamaguti algebra $\phi : L_1 \rightarrow L_2.$ 
 Since $(L_1[[t]],f_t,g_t)$ is a Lie Yamaguti algebra, we have 
\begin{align}\label{LYAD 1}
&f_t(x,x)=0, g_t(x,x,y)=0,\\
& \sum_{\circlearrowleft(x,y,z)} (f_t(f_t(x,y),z)+g_t(x,y,z))=0\\
&\sum _{\circlearrowleft (x,y,z,w)}g_t(f_t(x,y),z,w)=0\\
&g_t(x,y,f_t(z,w))=f_t(g_t(x,y,z),w)+f_t(z,g_t(x,y,z))\\
& g_t(x,y,g_t(z,w,u))=g_t(g_t(x,y,z),w,u)+g_t(z,g_t(x,y,w),u)+g_t(z,w,g_t(x,y,u))\label{LYAD 2}
\end{align}
for all $x,y,z,w,u \in L_1.$\\
Now since $f_t=\sum _{i \geq 0}f_it^i$ and  $g_t=\sum _{i \geq 0}g_it^i$, where $f_0=[~,~],~g_0=\{~,~,~\}$.
Thus, equating the coefficients of $t^n$, $n\geqslant 0$ from both sides of the Equations $(\ref{LYAD 1})-(\ref{LYAD 2})$ we get,
\begin{align}\nonumber
&f_n(x,x)=0\\\nonumber
& g_n(x,x,y)=0\\\label{LK1}
&\sum_{\substack{i+j=n\\i,j \geqslant 0}}\sum _{\circlearrowleft (x,y,z)} (f_i(f_i(x,y),z)+g_n(x,y,z))=0\\\label{LK2}
&\sum_{\substack{i+j=n\\i,j \geqslant 0}}\sum _{\circlearrowleft (x,y,z)}
(g_i(f_i(x,y),z,w))=0\\\label{LK3}
&\sum_{\substack{i+j=n\\i,j \geqslant 0}}(g_i(x,y,f_i(z,w))-f_i(g_i(x,y,z),w)-f_i(z,g_j(x,y,w)))=0\\\label{LK4}
& \sum _{\substack{i+j=n\\i,j \geqslant 0}}(g_i(x,y,g_j(z,w,u))-g_i(g_j(x,y,z),w,u)-g_i(z,g_i(x,y,w),u)-g_i(z,w,g_j(x,y,u)))=0
\end{align}
where $x,y,z,w,u \in L_1.$\\
Now putting $n=1$ in the Equations (\ref{LK3})-(\ref{LK4})we get
\begin{align*}
&\{x,y,f_1(u,v)\}+g_1(x,y,[u,v])-[g_1(x,y,u),v]\\
&-f_1(\{x,y,u\},v)-(z,g_1(x,y,z))-f_1(z,\{x,y,u\})=0;\\
&\mbox{and}\\
&\{x,y,g_1(u,v,w)\}+g_1(x,y,\{u,v,w\})-\{g_1(x,y,u),v,w\}\\
& -g_1(\{x,y,u\},v,w)-\{u,g_1(x,y,v),w\}-g_1(u,\{x,y,v\},w)\\
&-\{u,v,g_1(x,y,w)\}-g_1(u,v,\{x,y,w\})-\{u,v,g_1(x,y,w)\}=0.\\
\end{align*}
Therefore, we have $(\delta_I,\delta_{II})(f_1,g_1)=(0,0)$, i.e., $\delta^{'} (f_1,g_1)=0$

Now using the fact $(L_2[[t]],f^{'}_t,g^{'}_t)$ is a Lie-Yamaguti algebra we get similarly as above,
$(\delta_I,\delta_{II})(f_1^{'},g_1^{'})=(0,0)$, i.e., $\delta^{''} (f_1^{'},g_1^{'})=0.$

Now, since $\phi_t : L_1[[t]] \rightarrow L_2[[t]] $ is morphism of Lie-Yamaguti algebra, we have 
\begin{align}\label{LYAD 3}
&\phi_t(f_t(x,y))=f_t^{'}(\phi_t(x),\phi_t(y))\\ \label{LYAD 4}
&\phi_t (g_t(x,y,z))=g_t^{'}(\phi_t(x),\phi_t(y),\phi_t(z)) \\ \nonumber
\end{align}
for all $x,y,z \in L_1.$\\
Now the Equations (\ref{LYAD 3}) and (\ref{LYAD 4}) reduces to 
\begin{align*}
& \sum _{\substack{i+j=n\\i,j \geqslant 0}}\phi_i( f_j(x,y))=\sum_{\substack{i+j+k=n\\i,j,k \geq 0}}f_i^{'}(\phi_j(x), \phi_k(y))\\
& \sum_{\substack{i+j=n\\i,j \geqslant 0}}\phi_i (g_j(x,y,z))=\sum _{\substack{i+j+p+q=n\\i,j,p,q \geq 0}}g_i^{'}(\phi_j(x),\phi_p(y),\phi_q(z))
\end{align*}
Now putting $n=1$, the above equations reduces to 
\begin{align*}
\phi_1(f_0(x,y))+\phi(f_1(x,y))=f_1^{'}(\phi(x),\phi(y))+f_0^{'}(\phi_1(x),\phi(x))+f_0^{'}(\phi(x),\phi_1(y)).
\end{align*}
and
\begin{align*}
\phi_1(g_0(x,y,z))+\phi(g_1(x,y,z))&=g_1^{'}(\phi(x),\phi(y),\phi(z))+g_0^{'}(\phi_1(x),\phi(y),\phi(z))\\
&+g_0^{'}(\phi(x),\phi_1(y),\phi(z))+g_0^{'}(\phi(x),\phi(y),\phi_1(z)).
\end{align*}
Now, the above equation reduces to\\
$\phi \circ f_1-f_1^{'}\phi-\delta_I(\phi_1)=0$\\
$ ~ \mbox{and}~\\
 \phi \circ g_1-g_1^{'}\phi-\delta_{II}(\phi_1)=0.$\\
Hence, $\phi \circ (f_1,g_1)-(f_1^{'},g_1^{'})\phi-\delta(\phi_1)=0$.\\
Thus, we have $\delta(f_1,g_1)=0$, $\delta(f_1^{'},g_1^{'})=0$ and $\phi \circ (f_1,g_1)-(f_1^{'},g_1^{'})\phi-\delta(\phi_1)=0.$\\
Hence, $d ((f_1,g_1),(f_1^{'},g_1^{'}),(\phi_1,\phi_1))=0.$\\
Therefore, $((f_1,g_1),(f_1^{'},g_1^{'}),(\phi_1,\phi_1)) \in Z^{(2,3)}_{mLYA}(\phi,\phi).$

 \begin{definition}
 Let $((f_t,g_t),(f_t^{'},g_t^{'}),(\phi_t,\phi_t))$ and $((F_t,G_t),(F_t^{'},G_t^{'}),(\Phi_t,\Phi_t))$ be two one parameter formal deformation of a morphism of Lie-Yamaguti algebra $\phi: L_1 \rightarrow L_2$. Then this two one-parameter deformation is called equivalent if there exists a formal isomorphisms $(\psi_t,\psi^{'}_t)$ such that
 $\psi_t(x) : (L_1[[t]],f_t,g_t) \rightarrow (L_1[[t]], F_t,G_t)$ and $\psi^{'}_t(x): (L_2[[t]],f^{'}_t,g^{'}_t )\rightarrow (L_2[[t]],F^{'}_t,G^{'}_t)$, where $\psi_i : L_1 \rightarrow L_1$ and $\psi^{'}: L_2 \rightarrow L_2$ are two $k-$ linear map (extended to be $k[[t]]$ linear )
 such that 
 $\Phi_t=\psi_t^{'}\circ \phi_t \circ \psi_t^{-1}.$

 \end{definition}
\begin{definition}
Any deformation of $\phi:L_1 \rightarrow L_2$, which is equivalent to $((f_0,g_0),(f_0^{'},g_0^{'}),(\phi,\phi))$ is called a trivial deformation, where $f_0,g_0$ are the brackets of $L_1$ and $f_0^{'},g_0^{'}$ are the brackets of $L_2.$
\end{definition} 
 
\begin{theorem}
Let  $((f_t,g_t),(f_t^{'},g_t^{'}),(\phi_t,\phi_t))$ and $((F_t,G_t),(F_t^{'},G_t^{'}),(\Phi_t,\Phi_t))$ be two equivalent one parameter formal deformation of morphism of Lie-Yamaguti algebra $\phi : L_1 \rightarrow L_2$. Then $((f_t,g_t),(f_t^{'},g_t^{'}),(\phi_t,\phi_t))$ and $((F_t,G_t),(F_t^{'},G_t^{'}),(\Phi_t,\Phi_t))$ are belongs to the same cohomology class in $H^{(2,3)}_{mLYA}(\phi,\phi).$
\end{theorem} 
 \begin{proof}
 Since $(\hat{f_t},\hat{g_t},\phi_t)$ and $(\hat{F_t},\hat{G_t},\Phi_t)$ are two equivalent deformation, there exists isomorphisms say $\psi_t: (L_1[[t]],f_t,g_t) \rightarrow (L_1[[t]],F_t,G_t)$ and $\psi'_t: (L_2[[t]],f_t',g_t') \rightarrow (L_2[[t]],F_t',G_t')$.

Therefore,

\begin{align}\label{Equivalence 1}
\psi_t \circ f_t (x,y)=F_t(\psi_t(x),\psi_t(y)), \psi_t \circ g_t (x,y,z)=G_t(\psi_t(x),\psi_t(y),\psi_t(z))
\end{align}
 for all $x,y,z \in L_1$ with  $\psi_0=id_{L_1}$ and 
 \begin{align}\label{Equivalence 2}
 \psi^{'}_t \circ f^{'}_t (x,y)=F^{'}_t(\psi^{'}_t(x),\psi^{'}_t(y)), \psi^{'}_t \circ g^{'}_t (x,y,z)=G^{'}_t(\psi^{'}_t(x),\psi^{'}_t(y),\psi^{'}_t(z))
 \end{align}
 for all $x,y,z \in L_2$ with $\psi^{'}_0=id_{L_2}.$

From the Equation (\ref{Equivalence 1}), we have
\[\sum _{i \geq 0}\psi_i (\sum_{j \geq 0} f_j(x,y)t^j)t^i= \sum _{i \geq 0}F_i(\sum _{p \geq 0}\psi_p(x)t^p,\sum_{q\geq 0} \psi_q(y)t^q)t^i\]
and

\[\sum_{i \geq 0} \psi_i (\sum_{j \geq 0} g_j(x,y,z)t^j)t^i= \sum _{i\geq 0} G_i(\sum_{p \geq 0} \psi_(x)t^,\sum_{q \geq 0} \psi_q(y)t^q,\sum _{l \geq 0}\psi_l(z)t^l)t^i\]
 Comparing the coefficients of $t$ from both sides of the above equation we get
\[ f_1(x,y)+\psi_1([x,y])=F_1(x,y)+[\psi_1(x),y]+[x,\psi_1(y)]\]
i.e.,
\[f_1(x,y)-F_1(x,y)=\delta_I\psi_1(x,y)\] 
 and
 \[g_1(x,y,z)+\psi_1(\{x,y,z\})=G_1(x,y,z)+\{\psi_1(x),y,z\}+\{x,\psi_1(y),z\}+\{x,y,\psi_1(z)\}\]
 i.e., \[g_1(x,y,z)-G_1(x,y,z)=\delta_{II}\psi_1(x,y,z).\]
 Therefore, $(f_1-F_1,g_1-G_1)=\delta(\psi_1,\psi_1).$
 Similarly, as above from  the Equation (\ref{Equivalence 2}), we get
  $(f_1^{'}-F_1^{'},g_1^{'}-G_1^{'})=\delta(\psi_1^{'},\psi_1^{'}).$

Now, we have $\Phi_t \circ \psi_t= \psi_t^{'} \circ \phi_t.$
 Expanding and equating the coefficients of $t^n$ from both sides of the above equation we get,
\begin{align*}
\sum_{\substack{i+j=n\\i,j \geqslant 0}} \Phi_i \circ \psi_j=\sum_{\substack{i+j=n\\i,j \geqslant 0}}\psi_i^{'}\circ \phi_j.
\end{align*}
 
Now putting $n=1$ in the above equation we get
\begin{align*}
&\Phi_0 \circ \psi_1+\Phi_1 \circ \psi_0=\psi_0^{'} \circ \phi_1+\psi_1^{'}\circ \phi_0\\
&\phi_1-\Phi_1=\Phi \circ \psi_1 -\psi_1^{'} \circ \phi
\end{align*}
 
 Now,
 \begin{align*}
 &((f_1,g_1),(f_1^{'},g_1^{'}),(\phi_1,\phi_1))-((F_1,G_1),(F_1^{'},G_1^{'}),(\Phi_1,\Phi_1))\\
 &=((f_1-G_1,g_1-F_1),(f_1^{'}-F_1^{'},g_1^{'}-G_1^{'}),(\phi_1-\Phi_1,\phi_1-\Phi_1))\\
 &=(\delta(\psi_1,\psi_1),\delta(\psi_1^{'},\psi_1^{'}),(\Phi \circ \psi_1 -\psi_1^{'} \circ \phi,\Phi \circ \psi_1 -\psi_1^{'} \circ \phi))\\
 &=d((\psi_1,\psi_1),(\psi_1^{'},\psi_1^{'}),(0,0))
 \end{align*}
 Hence, $(f_1,g_1),(f_1^{'},g_1^{'}),(\phi_1,\phi_1))$ and $((F_1,G_1),(F_1^{'},G_1^{'}),(\Phi_1,\Phi_1)$ are cohomologous and belongs to the same cohomology class  $B^{(2,3)}_{mLYA}(\phi,\phi).$

 \end{proof}

 \begin{theorem}
 Let $ ((f_t,g_t),(f_t',g_t'),(\phi_t,\phi_t))$ be a formal one parameter deformation of a morphism of Lie-Yamaguti algebra $\phi : L_1 \rightarrow L_2$, with infinitesimal $((f_1,g_1),(f_1^{'},g_1^{'}),(\phi_1,\phi_1))$. Let $((F_1,G_1),(F_1^{'},G_1^{'}),(\Phi_1,\Phi_1))$ be a $(2,3)-$ cocyle cohomologous to $((f_1,g_1),(f_1^{'},g_1^{'}),(\phi_1,\phi_1))$, then there exists a deformation equivalent to $ ((f_t,g_t),(f_t^{'},g_t^{'}),(\phi_t,\phi_t))$ with infinitesimal $((F_1,G_1),(F_1^{'},G_1^{'}),(\Phi_1,\Phi_1))$.
 \end{theorem}
\begin{proof}
Since $((f_1,g_1),(f_1^{'},g_1^{'}),(\phi_1,\phi_1))$ and $((F_1,G_1),(F_1^{'},G_1^{'}),(\Phi_1,\Phi_1))$ are cohomologous, we have
\begin{align*}
((F_1,G_1),(F_1^{'},G_1^{'}),(\Phi_1,\Phi_1))=((f_1,g_1),(f_1^{'},g_1^{'}),(\phi_1,\phi_1))+(\delta(\lambda_1,\lambda_1)+\delta(\lambda_2,\lambda_2),(m,m))
\end{align*}
where $\lambda_1 \in C^1(L_1,L_1)$, $\lambda_2 \in C^1(L_2,L_2)$ and $(m,m)\in Z(L_1,L_2)$.\\
Let $\psi_t: L_1[[t]] \rightarrow L_1[[t]]$ and $\psi_t^{'}: L_2[[t]] \rightarrow L_2[[t]]$ be the formal isomorphisms obtained by extending linearly the map $\psi_t(a)=a+\lambda_1(a)t$, $a\in L_1$ to $L_1[[t]]$
and $\psi_t^{'}(a)=b+\lambda_2(b)t$, $b\in L_2$ to $L_2[[t]]$ respectively.
Then, the inverse maps are given by
\begin{align*}
& \psi_t^{-1}(a)=a+\sum_{i \geq 0}(-1)^i\lambda_1^{i}(a)t^i,~ a\in L_1\\
&{\psi_t^{'}}^{-1}(b)=b+\sum_{i \geq 0}(-1)^i\lambda_2^{i}(b)t^i,~b\in L_2.
\end{align*}
Now define the maps $(F_t,G_t)$, $(f_t^{'},G_t^{'})$ and $\Phi_t$ by
\begin{align}\label{LYADE 1}
&F_t(x,y)=\psi_t^{-1}(f_t(\psi_t(x),\psi_t(y))),~ x,y \in L_1\\\label{LYADE 2}
&G_t(x,y,z)=\psi_t^{-1}(g_t(\psi_t(x),\psi_t(y),\psi_t(z))),~x,y,z \in L_1\\\label{LYADE 3}
&F_t^{'}(x,y)={\psi_t^{'}}^{-1}(f_t^{'}(\psi_t^{'}(x),\psi_t^{'}(y))),~ x,y \in L_2\\\label{LYADE 4}
&G_t^{'}(x,y,z)={\psi_t^{'}}^{-1}(g_t^{'}(\psi_t^{'}(x),\psi_t^{'}(y),\psi_t^{'}(z))),~x,y,z \in L_2\\\label{LYADE 5}
&\Phi_t=\psi_t^{'} \circ \phi_t \circ \psi_t^{-1}.
\end{align}
Note that $((F_t,G_t),(F_t^{'},G_t^{'}),(\Phi_t,\Phi_t))$ is equivalent to $((f_t,g_t),(f_t^{'},g_t^{'}),(\phi_t,\phi_t))$.

Now from the Equation (\ref{LYADE 1}) we get the expression of $F_t(x,y)$ as

 \[\psi_t^{-1}(f_t(\psi_t(x),\psi_t(y)))=\psi_t^{-1}(\sum_{i \geq 0}(f_i(x+\lambda_1(x)t, y+\lambda_1(y)t)t^i),~ x,y \in L_1.\]

Expanding the above equation we get the coefficient of $t$ in $F_t(x,y)$ as
\[f_1(x,y)+[\lambda_1(x),y]+[x,\lambda_1(y)]-\lambda_1([x,y])=f_1(x,y)+\delta_I \lambda_1(x,y)=F_1(x,y)\]
i.e., $F_1-f_1=\delta_I\lambda_1.$

Again, from the Equation ({\ref{LYADE 2}}) we get the expression of $G_t(x,y,z)$ as
\[\psi_t^{-1}(g_t(\psi_t(x),\psi_t(y),\psi_t(z)))=\psi_t^{-1}(\sum_{i \geq 0}(g_i(x+\lambda_1(x)t, y+\lambda_1(y)t,z+\lambda_1(z))t^i), ~ x,y,z \in L_2.\]

Expanding the above equation we get the coefficient of $t$ in $G_t(x,y,z)$ as
\begin{align*}
&g_1(x,y,z)+\{\lambda_1(x),y,z\}+\{x,\lambda(y),z\}+\{x,y,\lambda_1(z)\}-\lambda_1(\{x,y,z\})\\
&=g_1(x,y,z)+\delta_{II}\lambda_1(x,y,z)=G_1(x,y,z)
\end{align*}
i.e., $G_1-g_1=\delta_{II}\lambda_1.$\\
Similarly from the Equations (\ref{LYADE 3}), (\ref{LYADE 4}) we get, $F_1^{'}-f_1^{'}=\delta_I\lambda_2, $ and $G_1^{'}-g_1^{'}=\delta_{II}\lambda_2.$

Again, from the  Equation (\ref{LYADE 5}) we get,
\begin{align*}
&\Phi_t(a)=\psi_t^{'}(\phi_t(a+\sum_{i\geq 0}(-1)^i \lambda_1^{i}(a)t^i))\\
&=\phi_t(a+\sum_{i\geq 0}(-1)^i \lambda_1^{i}(a)t^i)+\lambda_2(a+\sum_{i\geq 0}(-1)^i \lambda_1^{i}(a)t^i)t,~~a\in L_1.
\end{align*}
From the above equation we get the coefficient of $t$ in $\Phi_t$ as\\
$\phi_1 -\phi \circ \lambda_1+ \lambda_2 \phi=\Phi_1$ i.e., $\Phi_1-\phi_1=\lambda_2\phi -\phi \circ \lambda_1.$\\
Now as $d((\delta_I\lambda_1,\delta_{II}\lambda_1),(\delta_I \lambda_2,\delta_{II}\lambda_2),(\lambda_2\phi -\phi \circ \lambda_1,\lambda_2\phi -\phi \circ \lambda_1))=0$.\\
 Hence, 
$((F_1,G_1),(F_1^{'},G_1^{'}),(\Phi_1,\Phi_1))$ is the infinitesimal of $((F_t,G_t),(F_t^{'},G_t^{'}),(\Phi_t,\Phi_t))$.

\end{proof}
\begin{definition}
A Lie-Yamaguti algebra morphism $\phi:L_1 \rightarrow L_2$ is said to be rigid if and only if every deformation of $\phi$ is trivial.
\end{definition}

\begin{theorem}
A non-trivial deformation of a morphism of Lie-Yamaguti algebra is equivalent to a deformation whose $n$-infinitesimal is not a coboundary for some $n\geq 1.$
\end{theorem}

\begin{proof}
Let $((f_t,g_t),(f_t^{'},g_t^{'}),(\phi_t,\phi_t))$ be a deformation of $\phi: L_1 \rightarrow L_2$ with $n$-infinitesimal $((f_n,g_n),(f_n^{'},g_n^{'}),(\phi_n,\phi_n))$ for some $n\geq 1.$ Assume that there exists a $(2,3)$- cochain $((\psi,\psi),(\psi^{'},\psi^{'}),(m,m)) \in C^{(2,3)}_{mLYA}(\phi,\phi)$ with \\
 $d((\psi,\psi),(\psi^{'},\psi^{'}),(m,m))=((f_n,g_n),(f_n^{'},g_n^{'}),(\phi_n,\phi_n)).$

Now we may assume that $(m,m)=(0,0)$ as $d((\psi,\psi),(\psi^{'},\psi^{'}),(m,m))=d((\psi,\psi),(\psi^{'},\psi^{'})+\delta(m,m),(0,0)).$

Therefore, we have
 \begin{align*}
&\delta_I(\psi)=f_n ~;~\delta_{II}(\psi)=g_n ,\\
&\delta_I(\psi^{'})=f_n^{'} ~;~\delta_{II}(\psi^{'})=g_n^{'}  ~ \mbox{and}~\\
& \phi_n=\phi \circ \psi-\psi^{'}\phi.
\end{align*}
let $\psi_t : L_1[[t]] \rightarrow L_1[[t]]$ be formal isomorphism defined by $\psi_t =id_{L_1}+\psi t^{n}$ \\
and\\
 $\psi_t^{'} : L_2[[t]] \rightarrow L_2[[t]]$ be formal isomorphism defined by $\psi_t^{'} = id_{L_2}+\psi^{'} t^{n}.$
Now, we have a deformation $((F_t,G_t),(F_t^{'},G_t^{'}),(\Phi_t,\Phi_t))$ where 

\begin{align}\label{LYADEN 1}
&F_t(x,y)=\psi_t^{-1}(f_t(\psi_t(x),\psi_t(y))),~ x,y \in L_1\\\label{LYADEN 2}
&G_t(x,y,z)=\psi_t^{-1}(g_t(\psi_t(x),\psi_t(y),\psi_t(z))),~x,y,z \in L_1\\\label{LYADEN 3}
&F_t^{'}(x,y)={\psi_t^{'}}^{-1}(f_t^{'}(\psi_t^{'}(x),\psi_t^{'}(y))),~ x,y \in L_2\\\label{LYADEN 4}
&G_t^{'}(x,y,z)={\psi_t^{'}}^{-1}(g_t^{'}(\psi_t^{'}(x),\psi_t^{'}(y),\psi_t^{'}(z))),~x,y,z \in L_2\\
\label{LYADEN 5}
&\Phi_t=\psi_t^{'} \circ \phi_t \circ \psi_t^{-1}.
\end{align}

Expanding both sides of the Equation (\ref{LYADEN 1}) and equating the coefficients of $t^i,~ i \leq n$ we get $F_i=0,~ 1\leq i \leq n-1$ and $F_n=f_n-\delta_I(\psi)=0.$\\ 
Similarly using the Equations (\ref{LYADEN 2})-(\ref{LYADEN 5}) we can show that $G_i=0,~F_i^{'}=0,~G_i^{'}=0,~\Phi_i=0$ for all $1 \leq i \leq n.$ Therefore, the given deformation $((f_t,g_t),(f_t^{'},g_t^{'}),(\phi_t,\phi_t))$ is equivalent to $((F_t,G_t),(F_t^{'},G_t^{'}),(\Phi_t,\Phi_t))$ with 
$((F_i,G_i),(F_i^{'},G_i^{'}),(\Phi_i,\Phi_i))=0$ for $1 \leq i \leq n.$ Hence, we can repeat the argument to kill off any infinitesimal which is coboundary, so the process must stop if the deformation is non trivial.

\end{proof}

\begin{cor}
If $H^{(2,3)}_{mLYA}(\phi,\phi)=0$, then $\phi: L_1 \rightarrow L_2$ is rigid.
\end{cor}

\section{Abelian extension of morphisms of Lie-Yamaguti algebras}\label{sec6}
In this section we discuss the abelian extensions of morphism of Lie-Yamaguti algebras using the theory of abelian extension of Lie-Yamaguti algebra (\cite{ZL}) and motivated by the abelian extensions of morphisms of Lie algebra (\cite{das}). We show that isomorphism classes of abelian extension of morphism of Lie-Yamaguti algebra has one to one correspondence with the ‘simple’ $(2,3)$- cohomology group. \\

Note that any vector space $V$ can be regarded as a Lie-Yamaguti algebra where $[a,b]=0,~\{a,b,c\}=0$ for all $a,b,c\in V.$ Thus, any linear operator $\phi : V \rightarrow W$, where $V,~W$ are vector spaces , have a structure of morphism of Lie-Yamaguti algebra.\\

Let $\phi: L_1 \rightarrow L_2$ be a morphism of Lie-Yamaguti algebra with a representation $(V,W,\psi).$ A $(2n,2n+1)$-coboundary in $B_{mLYA}^{(2n,2n+1)}(\phi,\psi)$ is called simple if it has the form $d_{mLYA}(\alpha,\beta,0)$ for some $(\a,\b,0) \in C_{mLYA}^{(2n-2,2n-1)}(\phi,\psi).$ Then the space of all simple $(2n,2n+1)$-coboundaries are denoted by $B_{mLYA,s}^{(2n,2n+1)}(\phi,\psi)$ and the quotients $H_{mLYA,s}^{(2n,2n+1)}(\phi,\psi) :=\frac{Z_{mLYA}^{(2n,2n+1)}(\phi,\psi)}{B_{mLYA,s}^{(2n,2n+1)}(\phi,\psi)}$, for $n\geq 1$ are called the simple cohomology groups of the morphism of Lie-Yamaguti algebra $\phi :L_1 \rightarrow L_2$ with coefficients in $(V,W,\psi).$

\begin{definition}
An abelian extension of a morphism of Lie-Yamaguti algebra $\phi : L_1 \rightarrow L_2$  by $(V,W,\psi)$   is a linear map between the vectors spaces $\psi: V \rightarrow W$ and is a short exact sequence of morphisms of Lie-Yamaguti algebra 
\[\begin{tikzcd}
	0 & V & {\hat{L_1}} & {L_1} & 0 \\
	0 & W & {\hat{L_2}} & {L_2} & 0
	\arrow[from=1-1, to=1-2]
	\arrow["i", from=1-2, to=1-3]
	\arrow["p", from=1-3, to=1-4]
	\arrow[from=1-4, to=1-5]
	\arrow["\psi", from=1-2, to=2-2]
	\arrow["{\hat{\phi}}", from=1-3, to=2-3]
	\arrow["\phi", from=1-4, to=2-4]
	\arrow[from=2-1, to=2-2]
	\arrow["{\bar{i}}"', from=2-2, to=2-3]
	\arrow["{\bar{p}}"', from=2-3, to=2-4]
	\arrow[from=2-4, to=2-5]
\end{tikzcd}\]

In this case we say that $\hat{\phi} :\hat{L_1}\rightarrow \hat{L_2}$ is an abelian extension of the morphism of Lie-Yamaguti algebra $\phi : L_1 \rightarrow L_2$ by $(V,W,\psi).$
\end{definition}
\begin{definition}
Two abelian extension $\hat{\phi} :\hat{L_1}\rightarrow \hat{L_2}$ and $\hat{\phi^{'}} :\hat{L_1^{'}}\rightarrow \hat{L_2^{'}}$ are said to be isomorphic if there exists an isomorphism $(\alpha,\beta)$ of  morphism Lie-Yamaguti algebra from 
$\hat{\phi} :\hat{L_1}\rightarrow \hat{L_2}$ to $\hat{\phi^{'}} :\hat{L_1^{'}}\rightarrow \hat{L_2^{'}}$ so that the following diagram is commutative

\begin{tikzcd}
	0 & V && {\hat{L_1}} && {L_1} & 0 \\
	0 && V && {\hat{L_1^{'}}} && {L_1} & 0 \\
	0 & W && {\hat{L_2}} && {L_2} && 0 \\
	& 0 & W && {\hat{L_2^{'}}} && {L_2} & 0
	\arrow[Rightarrow, no head, from=1-2, to=2-3]
	\arrow[from=2-3, to=4-3]
	\arrow[from=1-2, to=3-2]
	\arrow[Rightarrow, no head, from=3-2, to=4-3]
	\arrow[from=2-5, to=4-5]
	\arrow["\beta"', from=3-4, to=4-5]
	\arrow[Rightarrow, no head, from=1-6, to=2-7]
	\arrow[from=2-7, to=4-7]
	\arrow[from=1-6, to=3-6]
	\arrow[Rightarrow, no head, from=3-6, to=4-7]
	\arrow[from=1-1, to=1-2]
	\arrow[from=3-1, to=3-2]
	\arrow[from=4-2, to=4-3]
	\arrow[from=1-6, to=1-7]
	\arrow[from=2-7, to=2-8]
	\arrow[from=4-7, to=4-8]
	\arrow[from=4-3, to=4-5]
	\arrow[from=4-5, to=4-7]
	\arrow[from=2-3, to=2-5]
	\arrow[from=2-5, to=2-7]
	\arrow[from=2-1, to=2-3]
	\arrow[from=3-6, to=3-8]
	\arrow["{\bar{p}}"{pos=0.7}, from=3-4, to=3-6]
	\arrow["{\bar{i}}"{pos=0.7}, from=3-2, to=3-4]
	\arrow["i", from=1-2, to=1-4]
	\arrow["\alpha"', from=1-4, to=2-5]
	\arrow[from=1-4, to=3-4]
	\arrow["p", from=1-4, to=1-6]
\end{tikzcd}

\end{definition}
\begin{definition}
A section is a pair $(s,\bar{s})$ of linear maps $s : L_1 \rightarrow \hat{L_1}$ and $\bar{s}: L_2 \rightarrow \hat{L_2}$ such that $p \circ s=id_{L_1}, ~\bar{p}\circ \bar{s}=id_{L_2}.
$
\end{definition}

Let $\hat{\phi} :\hat{L_1}\rightarrow \hat{L_2}$ be an abelian extension of the morphism of Lie-Yamaguti algebra $\phi : L_1 \rightarrow L_2$ by $(V,W,\psi).$ Let $(s,\bar{s})$ be a section of this abelian extension. Now define a linear map  $\rho_V:L_1 \rightarrow End(V)$ and two bilinear maps $D_V : L_1 \times L_1 \rightarrow End(V),~\theta_V: L_1 \times L_1 \rightarrow End(V) $ by 
\begin{align*}
&\rho_V(x)v=[s(x),i(v)]_{\hat{L_1}},~ D_V(x,y)v=\{s(x),s(y),i(v)\}_{\hat{L_1}}\\
& \theta_V(x,y)v=\{i(v),s(x),s(y)\}_{\hat{L_1}}~~ \mbox{for all}~ x,y\in L_1,~v\in V.
\end{align*}
and also define 
a linear map  $\rho_W:L_2 \rightarrow End(W)$ and two bilinear maps $D_W : L_2 \times L_2 \rightarrow End(W),~\theta_W : L_2 \times L_2 \rightarrow End(W) $ by 
\begin{align*}
&\rho_W(a)w=[\bar{s}(a),\bar{i}(w)]_{\hat{L_2}},~ D_W(a,b)w=\{\bar{s}(a),\bar{s}(b),\bar{i}(w)\}_{\hat{L_2}}\\
& \theta_W(a,b)w=\{\bar{i}(w),\bar{s}(a),\bar{s}(b)\}_{\hat{L_2}}~~ \mbox{for all}~ a,b\in L_2,~w\in W.
\end{align*}

Now with the above definitions we have the following proposition.
\begin{proposition}\label{Abelian LYA 1}
If $(s,\bar{s})$ is a section then $(\rho_V,D_V,\theta_V)$ is a representation of $L_1$ on $V$ and $(\rho_W,D_W,\theta_W)$ is a representation of $L_2$ on $W.$

\end{proposition}
\begin{proof}
Now for any $x,y,z \in L_1$ and $v\in V$, we have\\
\textbf{1.}
\begin{align*}
&D_V(x,y)v-\theta_V(y,x)v+\theta_V(x,y)v+\rho_V([x,y])v-\rho_V(x)\rho_V(y)v+\rho_V(y)\rho_V(x)v\\
&=\{s(x),s(y),i(v)\}_{\hat{L_1}}-\{i(v),s(y),s(x)\}_{\hat{L_1}}+\{i(v),s(x),s(y)\}_{\hat{L_1}}\\
&+[[s(x),s(y)]_{\hat{L_1}},i(v)]_{\hat{L_1}}-[s(x),[s(y),i(v)]_{\hat{L_1}}]_{\hat{L_1}}+[s(y),[s(x),i(v)]_{\hat{L_1}}]_{\hat{L_1}}\\
&=\{s (x),s(y),i(v)\}_{\hat{L_1}}+\{s(y),i(v),s(x)\}_{\hat{L_1}}+\{i(v),s(x),s(y)\}_{\hat{L_1}}\\
&+[[s(x),s(y)]_{\hat{L_1}},i(v)]_{\hat{L_1}}+[[s(y),i(v)]_{\hat{L_1}},s(x)]_{\hat{L_1}}+[[i(v),s(x)]_{\hat{L_1}},s(y)]_{\hat{L_1}}\\
&=0~~~~~~~ (\mbox{since}~ \hat{L_1}~ \mbox{is a Lie-Yamaguti alebra.})
\end{align*}
\textbf{2.}
\begin{align*}
&D([x,y],z)v+D([y,z],x)v+D([z,x],y)v\\
&=\{s([x,y]_{\hat{L_1}}),s(z),i(v)\}_{\hat{L_1}}+\{s([y,z]_{\hat{L_1}}),s(x),i(v)\}_{\hat{L_1}}+\{s([z,x]_{\hat{L_1}}),s(y),i(v)\}_{\hat{L_1}}\\
&=\{[s(x), s(y)]_{\hat{L_1}},s(z),i(v)\}_{\hat{L_1}}+\{[s(y),s(z)]_{\hat{L_1}},s(x),i(v)\}_{\hat{L_1}}+\{[s(z),s(x)]_{\hat{L_1}},s(y),i(v)\}_{\hat{L_1}}\\
&=0~~~~~~~ (\mbox{since}~ \hat{L_1}~ \mbox{is a Lie-Yamaguti alebra.})
\end{align*}
The other axioms of representation of a Lie-Yamaguti algebra $L_1$ can be proved similarly as above.

Now for any $x,y\in L_1$ and $v\in V$ we have
\begin{align*}
&\psi(\rho _V(x)v)-\rho_W(\phi(x))\psi(v)\\
&=\hat{\phi}[s(x),i(v)]_{\hat{L_1}}-[\bar{s}\circ \phi(x),\bar{i}\circ \psi(v)]_{\hat{L_1}}\\
&=[\hat{\phi}\circ s(x)-\bar{s}\circ \phi(x), \hat{\phi} \circ i(v)-\bar{i}\circ \psi(v)]_{\hat{L_2}}\\
&=0,~ \mbox {as}~ \hat{\phi}\circ s(x)-\bar{s}\circ \phi(x), ~\hat{\phi} \circ i(v)-\bar{i}\circ \psi(v) \in \mbox{ker}~(\bar{p})=\mbox{im}~(\bar{i}).
\end{align*}
Again we have,
\begin{align*}
&\psi(D_V(x,y)v)-D_W(\phi(x),\phi(y))\psi(v)\\
&=\hat{\phi}\{s(x),s(y),i(v)\}_{\hat{L_1}}-\{\bar{s}\circ \phi(x),\bar{s}\circ \phi(y),\bar{i}\circ \psi(v)\}_{\hat{L_2}}\\
&=\{\hat{\phi} \circ s(x)-\bar{s}\circ \phi(x),\hat{\phi}\circ s(y)-\bar{s}\circ \phi(y),\hat{\phi} \circ i(v)-\bar{i}\circ \psi(v)\}\\
&=0
\end{align*}
Similar as above we get $\psi(\theta_V(x,y)v)-\theta_W(\phi(x),\phi(y))\psi(v)=0.$

 Therefore, $(\rho_V,D_V,\theta_V)$ is a representation of $L_1$ on $V.$
Also by the same argument as above we get $(\rho_W,D_W,\theta_W)$ is a representation of $L_2$ on $W.$

\end{proof}

\begin{proposition}
The representation of $L_1$ on $V$ and the representation of $L_2$ on $W$ induced from a section of an abelian extension is independent of the choice of section.
\end{proposition}

\begin{proof}
Let $(s,\bar{s})$ and $(s^{'},\bar{s^{'}})$ be two section.
Now $\rho_V(x)-\rho_V^{'}(x)=[s(x)-s^{'}(x),i(v)]_{\hat{L_1}}=0$ and  $\rho_W(a)-\rho_W^{'}(a)=[s(a)-s^{'}(a),i(w)]_{\hat{L_2}}=0,$ since $s(x)-s^{'}(x)\in \mbox{ker}~ p= \mbox{im}~i $ and $\bar{s}(a)-\bar{s^{'}}(a) \in \mbox{ker}~ \bar{p}= \mbox{im}~\bar{i}$ for all $x\in L_1,~a\in L_2,~v\in V,~w \in W.$ 

Similarly, $D_V(x,y)v-D_V^{'}(x,y)v=0$, $D_W(a,b)w-D_W^{'}(a,b)w=0$ and $\theta(x,y)v-\theta^{'}(x,y)v=0$, $\theta_W(a,b)w-\theta_W^{'}(a,b)w=0$ for all $x,y \in L_1,~a,b \in L_2,~v\in V,~w \in W.$ Therefore, the representation induced by the sections $(s,\bar{s})$ and $(s^{'},\bar{s^{'}})$ are same.
\end{proof}

Let $\phi: L_1 \rightarrow L_2$ be a morphism of Lie-Yamaguti algebra with a representation $(V,W,\psi)$. Now define $Ext(\phi,\psi)$ be the isomorphism classes of the abelian extensions of $\phi:L_1 \rightarrow L_2$ by the representation $(V,W,\psi)$. It is clear from the Proposition (\ref{Abelian LYA 1})  $Ext(\phi,\psi)$ is non empty.

\begin{lemma} 
Every abelian extension of  $\phi : L_1 \rightarrow L_2$ by $(V,W,\psi)$  gives a cohomology class in $H^{(2,3)}_{mLYA,s}(\phi,\psi).$
\end{lemma}
\begin{proof}
Let $\hat{\phi} :\hat{L_1}\rightarrow \hat{L_2}$ be an abelian extension of $\phi : L_1 \rightarrow L_2$ by $(V,W,\psi)$. Let $(s,\bar{s})$ be a section. Now define the maps

$\alpha_I :L_1 \times L_1 \rightarrow V,~ \beta : L_2 \times L_2 \rightarrow W, ~\gamma : L_1 \rightarrow W$ by 
\begin{align*}
&\alpha_I(x,y):=[s(x),s(y)]_{\hat{L_1}}-s([x,y]_{L_1}),\\
&\beta_I(a,b):=[\bar{s}(a),\bar{s}(b)]_{\hat{L_2}}-\bar{s}([a,b]_{L_2}) ~~~ \mbox{and}\\
& \gamma(x):=\hat{\phi}(s(x))-\bar{s}(\phi(s)),~~~ \mbox{for all }~ x,y \in L_1,~a,b \in L_2.
\end{align*}
and 
$\alpha_{II} :L_1\times L_1 \times L_1\rightarrow V,~ \beta : L_2 \times L_2 \times L_2 \rightarrow W$ by 
\begin{align*}
&\alpha_{II}(x,y,z):=\{s(x),s(y),s(z)\}_{\hat{L_1}}-s(\{x,y,z\}_{L_1})\\
&\beta_{II}(a,b,c):=\{\bar{s}(a),\bar{s}(b),\bar{s}(c)\}_{\hat{L_2}}-\bar{s}(\{a,b,c\}_{L_2})~~~ 
\end{align*}
Therefore, $((\alpha_I,\alpha_{II}),(\beta_I,\beta_{II}),(\gamma,\gamma)) \in C^{(2,3)}_{mLYA}(\phi,\psi).$ By the theory of abelian extension of Lie Yamaguti algebra (\cite{ZL}) we have $\delta(\alpha_I,\alpha_{II})=0, \delta(\beta_{I},\beta_{II})=0$ and a direct computation shows that $\phi\circ (\alpha_I,\alpha_{II})-(\beta_I,\beta_{II})\phi-\delta(\gamma,\gamma)=0.$ Hence, $((\alpha_I,\alpha_{II}),(\beta_I,\beta_{II}),(\gamma,\gamma)) \in Z^{(2,3)}_{mLYA}(\phi,\psi)$. Therefore,  $((\alpha_I,\alpha_{II}),(\beta_I,\beta_{II}),(\gamma,\gamma))$ corresponds to a cohomology class in $H^{(2,3)}_{mLYA,s}(\phi,\psi).$

\end{proof}

\begin{lemma}
The cohomology class induced from an abelian extension of  $\phi : L_1 \rightarrow L_2$ by $(V,W,\psi)$  does not depend on the choice of section.
\end{lemma}
\begin{proof}
Let $\hat{\phi} :\hat{L_1}\rightarrow \hat{L_2}$ be an abelian extension of a morphism of Lie-Yamaguti algebra $\phi : L_1 \rightarrow L_2$ by $(V,W,\psi).$ Let $(s,\bar{s})$ and $(s^{'},\bar{s^{'}})$ be two section of it. Now define $\eta: L_1 \rightarrow V$ by $\eta(x)=s(x)-s^{'}(x)$ for all $x \in L_1$ and $ \eta^{'}: L_2 \rightarrow W$ by $\eta^{'}(a)=\bar{s}(a)-\bar{s^{'}}(a)$ for all $a \in L_2.$

Now,\begin{align*}
&\alpha_{I}(x,y)\\
&=[s(x),s(y)]_{\hat{L_1}}-s([x,y]_{L_1})\\
&=[s^{'}(x)+\eta(x),s^{'}(y)+\eta(y)]_{\hat{L_1}}-(s^{'}([x,y])+\eta
([x,y]))\\
&=[s^{'}(x),s^{'}(y)]_{\hat{L_1}}-s^{'}([x,y]_{L_1})+[\eta(x),s^{'}(y)]_{\hat{L_1}}+[s^{'}(x),\eta(y)]_{\hat{L_1}}-\eta
([x,y]_{L_1})~( \mbox{as}~ [\eta(x),\eta(y)]=0)\\
&=\alpha_I^{'}(x,y)+\delta_I(\eta)(x,y)
\end{align*}
Similarly we have, $\alpha_{II}(x,y,z)=\alpha^{'}(x,y,z)+\eta_{II}(\eta)(x,y,z)$
Therefore, $(\alpha_I,\alpha_{II})-(\alpha^{'}_I,\alpha^{'}_{II})=(\delta_I(\eta),\delta_{II}(\eta)) \in Z^{(2,3)}(L_1,V).$\\
Similarly using $\eta^{'}$ we will get $(\beta_I,\beta_{II})-(\beta_I^{'}\beta_{II}^{'})\in Z^{(2,3)}(L_2,W)$ and $(\gamma,\gamma)-(\gamma^{'},\gamma^{'})\in Z(L_1,W).$
Therefore, $((\alpha_I,\alpha_{II}),(\beta_I,\beta_II),(\gamma,\gamma))$ and $((\alpha_I^{'},\alpha_{II}^{'}),(\beta_I^{'},\beta_II^{'}),(\gamma^{'},\gamma^{'}))$ are in the same cohomology class in $H_{mLYA,s}^{(2,3)}(\phi,\psi).$

\end{proof}
\begin{theorem}
 Isomorphic abelian extensions of  $\phi : L_1 \rightarrow L_2$ by $(V,W,\psi)$ gives rise to the same element in $H^{(2,3)}_{mLYA,s}(\phi,\psi)$.
\end{theorem}
\begin{proof}
Let $\hat{\phi} :\hat{L_1}\rightarrow \hat{L_2}$ and $\hat{\phi^{'}} :\hat{L_1^{'}}\rightarrow \hat{L_2^{'}}$ be two isomorphic abelian extensions and let $(\alpha, \beta)$ be the isomorphism between them. Now suppose $(s,\bar{s})$ be a section of of the abelian extension $\hat{\phi} :\hat{L_1}\rightarrow \hat{L_2}.$ Therefore,
$p^{'}\circ (\alpha \circ s)=p \circ s=id_{L_1}$ and $\bar{p^{'}}\circ (\beta \circ \bar{s})=\bar{p}\circ \bar{s}=id_{L_2}.$ Thus, $(\alpha \circ s,\beta \circ \bar{s})$ is a section of $\hat{\phi^{'}} :\hat{L_1^{'}}\rightarrow \hat{L_2^{'}}$. Now let $((\alpha_I^{'},\alpha_{II}^{'}),(\beta_I^{'},\beta_II^{'}),(\gamma^{'},\gamma^{'}))\in Z^{(2,3)}_{mLYA}$ be the $(2,3)$- cocycle corresponding to the abelian extension $\hat{\phi^{'}} :\hat{L_1^{'}}\rightarrow \hat{L_2^{'}}$ with section $(\alpha \circ s,\beta \circ \bar{s}).$ Then we have
\begin{align*}
\alpha_{I}^{'}(x,y)&=[\alpha \circ s(x),\alpha \circ s(y)]_{\hat{L_1}}-\alpha \circ s([x,y]_{L_1})\\
&=\alpha([s(x),s(y)]_{\hat{L_1}}-s([x,y]_{L_1}))\\
&=\alpha_I(x,y)~~(\mbox{as} ~\alpha_I|_{V}=id_V),~\mbox{for all}~ x,y\in L_1.
\end{align*}

Similarly, we will get $\alpha_{II}^{'}(x,y,z)=\alpha_{II}(x,y,z)$ for all $x,y,z \in L_1$;
\\ $\beta_I^{'}(a,b)=\beta_{I}(a,b)$ and $\beta_{II}^{'}(a,b,c)=\beta_{II}(a,b,c)$ for all $a,b,c \in L_2$ and $\gamma^{'}(x)=\gamma(x)$ for all $x \in L_1.$ Hence, $((\alpha_I,\alpha_{II}),(\beta_I,\beta_II),(\gamma,\gamma))=((\alpha_I^{'},\alpha_{II}^{'}),(\beta_I^{'},\beta_II^{'}),(\gamma^{'},\gamma^{'})).$ 

Thus, isomorphic abelian extension gives rise to the same element in $H^{(2,3)}_{mLYA,s}(\phi,\psi)$.

\end{proof}

\begin{theorem}
Every element in $H^{(2,3)}_{mLYA,s}(\phi,\psi)$ induces an abelian extension of $\phi : L_1 \rightarrow L_2$ by $(V,W,\psi)$. 
\end{theorem}
\begin{proof}
Let $((\alpha_I,\alpha_{II}),(\beta_I,\beta_II),(\gamma,\gamma)) \in Z^{(2,3)}_{mLYA}(\phi,\psi)$ be a $(2,3)$- cocycle. Now define $\hat{L_1}:= L_1 \oplus V,~\hat{L_2}=L_2 \oplus W$ and define $\hat{\phi}: \hat{L_1} \rightarrow \hat{L_2}$ by $\hat{\phi}(x,v)=(\phi(x),\psi(v)+\gamma(x))$, for all $(x,v) \in \hat{L_1}.$ Now define the bilinear and trilinear brackets on $\hat{L_1}$ and $\hat{L_2}$ by
\begin{align*}
&[(x,v_1),(y,v_2)]_{\hat{L_1}}:=([x,y]_{L_1},\rho_V(x)v_2-\rho_V(y)v_1+\delta_I(x,y))\\
&\{(x,v_1),(y,v_2),(z,v_3)\}_{\hat{L_1}}:=(\{x,y,z\}_{L_1},\theta_V(y,z)v_1-\theta_V(x,z)v_2+D_V(x,y)v_3+\delta_{II}(x,y,z))\\
&\mbox{for all}~ x,y,z \in L_1,~ v_1,v_2,v_3\in V.\\
&\mbox{and}\\
&[(a,w_1),(b,w_2)]_{\hat{L_2}}:=([a,b]_{L_2},\rho_V(a)w_2-\rho_V(b)w_1+\delta_I(a,b))\\
&\{(a,w_1),(b,w_2),(c,w_3)\}_{\hat{L_2}}:=(\{a,b,c\}_{L_2},\theta_V(b,c)w_1-\theta_V(a,c)w_2+D_V(a,b)w_3+\delta_{II}(a,b,c))\\
&\mbox{for all}~ a,b,c \in L_2,~ w_1,w_2,w_3\in W.
\end{align*}
Then  $\hat{L_1}$, $\hat{L_2}$ are Lie-Yamaguti algebra with the above defined brackets (for proof see \cite{ZL}).
It is easy to show that
 $\hat{\phi} :\hat{L_1}\rightarrow \hat{L_2}$ is a morphisms of Lie-Yamaguti algebra
and  $\hat{\phi} :\hat{L_1}\rightarrow \hat{L_2}$ is an abelian extension of $(L_1,L_2,\phi)$ by $(V,W,\psi)$ (For details see \cite{ZL}).
\end{proof}
\begin{lemma}
Cohomologous elements in $H^{(2,3)}_{mLYA,s}(\phi,\psi)$ gives isomorphic abelian extensions of $\phi : L_1 \rightarrow L_2$ by $(V,W,\psi)$.
\end{lemma}
\begin{proof}

Let $((\alpha_I,\alpha_{II}),(\beta_I,\beta_II),(\gamma,\gamma))$, $((\alpha_I^{'},\alpha_{II}^{'}),(\beta_I^{'},\beta_II^{'}),(\gamma^{'},\gamma^{'})) \in Z^{(2,3)}_{mLYA}(\phi,\psi)$ such that
\begin{align*}
&((\alpha_I,\alpha_{II}),(\beta_I,\beta_II),(\gamma,\gamma))-((\alpha_I^{'},\alpha_{II}^{'}),(\beta_I^{'},\beta_II^{'}),(\gamma^{'},\gamma^{'}))\\
&=d_{mLYA}((\xi,\xi),(\xi^{'},\xi^{'}),(0,0)) ~\mbox{for some}~\xi \in C^1(L_1,V),~\xi^{'}\in C^2(L_2,W).
\end{align*}
Then the corresponding induced abelian extensions  $\hat{\phi} :\hat{L_1}\rightarrow \hat{L_2}$ and  $\hat{\phi^{'}} :\hat{L_1^{'}}\rightarrow \hat{L_2^{'}}$ respectively are isomorphic via the map $(\alpha,\beta)$ where $\alpha : \hat{L_1} \rightarrow \hat{L_1{'}}$ and $\beta : \hat{L_2} \rightarrow \hat{L_2^{'}}$ defined by $\alpha(x,v)=(x,v+\xi(x))$, $\beta(a,w)=(a,w+\xi^{'}(a))$, $x \in L_1,~a\in L_2,~v \in V,~w \in W$ respectively.

\end{proof}
From the above Lemmas and Theorems, we have the following main result of this section:
\begin{theorem}
Let $\phi: L_1 \rightarrow L_2$ be a morphism of Lie-Yamaguti algebra with a representation $(V,W,\psi)$. Then there exists a one-to-one correspondence between $Ext(\phi,\psi)$ and the $(2,3)$-cohomology group $H^{(2,3)}_{mLYA,s}(\phi,\psi).$
\end{theorem}

\section{Discussion about the study of a more general problem}\label{sec7}
Suppose $\bf{LYA}$ is the category of all Lie-Yamaguti algebras over a field $k$. Let $I= \lbrace i, j, k,\ldots \rbrace$ be a partially ordered set. One can think of $I$ as a category whose set of objects is $I$ itself and there is a unique morphism $i \to j$, whenever $i\leq j$. A diagram of Lie-Yamaguti algebras over $I$ is a contravariant functor $\mathfrak{L} : I \to \bf{LYA}$. The diagram is called finite if $I$ is a finite set. In this paper, we have considered a special case of this general form with $I= \lbrace i, j\rbrace$. Therefore, in this paper, our diagram is nothing but a single morphism of Lie-Yamaguti algebras. One can study the cohomology and deformation theory of a diagram of Lie-Yamaguti algebras generalizing our case of single morphism and following \cite{GS83}. 

{\bf Acknowledgements:} The second/corresponding author is supported by the Science and Engineering Research Board (SERB), Department of Science and Technology (DST), Govt. of India (Grant Number- CRG/2022/005332).

\renewcommand{\refname}{REFERENCES}


\begin{thebibliography}{99}


\bibitem{das}
A. Das, {\it Cohomology of Lie Algebra Morphism Triples and Some Applications}, \emph{Math Phys Anal Geom} 26, 26 (2023).


\bibitem{Duffin}
R. J. Duffin, {\it On the characteristic matrices of covariant systems}, \emph{ Physical Review}, Vol. 54 (1938), Page 1114. 

%

\bibitem{G1}
M. Gerstenhaber, {\it The Cohomology Structure of an Associative Ring}, \emph{Ann. of  Math.}, Vol.78 (1963), 267-288.

\bibitem{G2}
M. Gerstenhaber, {\it On the Deformations of Rings and Algebras}, \emph{Ann. of  Math.}, Vol.79 (1964), 59-107.

\bibitem{G3}
M. Gerstenhaber, {\it On the Deformations of Rings and Algebras}, \emph{Ann. of  Math.}, Vol.84 (1966), 1-19.

\bibitem{G4}
M. Gerstenhaber, {\it On the Deformations of Rings and Algebras}, \emph{Ann. of  Math.}, Vol.88 (1968), 1-34.

\bibitem{G5}
M. Gerstenhaber, {\it On the Deformations of Rings and Algebras}, \emph{Ann. of  Math.}, Vol.99 (1974), 257-276.

\bibitem{GS83}
M. Gerstenhaber, S.D. Schack, {\it On the deformation of algebra morphisms and diagrams}, \emph{Transactions of the American Mathematical Society}, 279 (1983) 1-50.

\bibitem{Goswami}
 S. Goswami, {\it Formal one parameter deformation of Lie-Yamaguti algebras}, \emph{arXiv: 2308.03655v1}.


\bibitem{NJ}
N. Jacobson, {\it Lie and Jordan triple systems}, \emph{Amer. J. Math. Soc.}, Vol. 71 (1949) 149-170.

\bibitem{MK75}
M. Kikkawa, {\it Geometry of Homogeneous Lie loops}, \emph{Hiroshima Math. J.}, Vol. 5 (1975), 141-179.

\bibitem{KW}
M. K. Kinyon and A. Weinstein, {\it Leibniz algebras, Courant algebroids,
and multiplications on reductive homogeneous spaces}, \emph{American J. Math.,} Vol. 123  (2001), 525-550.


%
%
\bibitem{LCM} 
J. Lin, Y. L. Chen, and Yao Ma, {\it  On the Deformation of Lie Yamaguti algebras}, \emph{Acta Mathematica Sinica}, Vol. 31 (2015), 938-946.

\bibitem{Loday1}
J. -L. Loday, {\it Une version non commutative des alg$\grave{e}$bres de Lie: les alg$\grave{e}$bres de Leibniz}, \emph{Enseign. Math.}, Vol. 39 (1993), 269-293.




\bibitem{ashish}
A. Mandal, {\it Deformation of Leibniz algebra morphisms}. \emph{Homology Homotopy Appl.} 9 (2007), 439–450.

%
%
%
\bibitem{KN}
K. Nomizu, {\it Invariant affine connections on homogeneous spaces}, \emph{Amer. J. Math} Vol. 76 (1954), 33-65.


	
\bibitem{KY} 
K. Yamaguti, {\it On the Lie triple systems and its generalization}, \emph{J. Sci. Hiroshima Univ. Ser. A}, Vol. 21, (1957), 155-160.

\bibitem{KY-LTS}
K. Yamaguti, {\it On the cohomology space of Lie triple systems}, \emph{ Kumamoto J. Sci. A.}, Vol. 5 (1960), 44–52.


\bibitem{KY-cohomology}
K. Yamaguti, {\it On cohomology groups of General Lie Triples systems}, \emph{Kumamoto J. Sci.}, Series A, Vol. 8, No. 4  (1969), 135-146.

\bibitem{ZT}
Zhang T, {\it Notes on cohomologies of Lie triple systems}, \emph{J. Lie Theory}, Vol. 24 (2014), 909–929.
	
\bibitem{ZL}
Tao Zhang and Juan Li, {\it Deformation and Extensions of Lie-Yamaguti algebras}, \emph{Linear and Multilinear Algebra}, Vol. 63 (2015), 2212-2231.

\end{thebibliography}
\end{document}